\documentclass[11pt]{article}
\usepackage{fullpage}
\usepackage{amsmath,amssymb,amsthm,amsfonts}
\usepackage{mathrsfs}
\usepackage{xspace}
\usepackage{algorithm}
\usepackage{algorithmic}
\usepackage{bm}
\usepackage{color}
\usepackage{multirow}
\usepackage{hhline}
\usepackage[colorlinks,citecolor={black},urlcolor={black},linkcolor={black}]{hyperref}
\usepackage[small,bf]{caption}
\usepackage{url}

\newcommand\ceil[1]{\lceil #1 \rceil}
\newcommand\M{\mathcal{M}}
\renewcommand{\vec}[1]{\boldsymbol{\mathbf{#1}}}
\newcommand\peel{$\mathcal{M}^{\mathrm{Peel}}$\xspace}

\newcommand\oneshot{\mathcal{M}^{\mathrm{os}}}

\newcommand\bhc{\alpha^{\textnormal{\tiny BH}}}

\newcommand\rhs{RHS\xspace}

\newcommand\iid{i.\,i.\,d.\xspace}

\newcommand{\R}{\mathbb{R}}
\newcommand{\ep}[1]{\ifthenelse{\equal{#1}{1}}{\mathrm{e}}{\mathrm{e}^{#1}}}

\def\prob{\mathbb{P}}
\let\P\prob
\def\lap{\operatorname{Lap}}

\def\argmin{\operatorname{argmin}}
\def\fdr{\operatorname{FDR}}
\def\fdp{\operatorname{FDP}}

\let\e\varepsilon
\renewcommand\epsilon{\varepsilon}

\def\d{\mathrm{d}}

\newcommand{\remove}[1]{}

\def\RNMa{Report Noisy Max\xspace}
\def\RNMi{Report Noisy Min\xspace}

\def\BHq{BHq\xspace}
\def\BHsmall{BHq$_{\textnormal{sd}}$\xspace}
\def\dba{D}
\def\dbb{D'}
\def\dbx{x}
\def\dby{y}
\def\eps{\varepsilon}

\newcommand{\E}{\mathbb{E}}

\def\wO{\widetilde{O}}
\newcommand\dbax{D}
\newcommand\dbby{D'}

\newtheorem{theorem}{Theorem}
\newtheorem{othertheorem}{othertheorem}[section]
\newtheorem{lemma}[othertheorem]{Lemma}
\newtheorem{corollary}[othertheorem]{Corollary}

\theoremstyle{definition}
\newtheorem{definition}[othertheorem]{Definition}
\theoremstyle{remark}

\theoremstyle{assumption}

\allowdisplaybreaks
\numberwithin{equation}{section}

\title{Private False Discovery Rate Control}
\author{Cynthia Dwork\textsuperscript{$\ast$} ~\and ~ Weijie Su\textsuperscript{$\dagger$} ~ \and ~ Li Zhang\textsuperscript{$\ddagger$}}

\begin{document}
\maketitle

{\centering
\vspace*{-0.3cm} 
\textsuperscript{$\ast$} Microsoft Research, Mountain View, CA 94043, USA\\
\textsuperscript{$\dagger$} Department of Statistics, Stanford University, Stanford, CA 94305, USA\\
\textsuperscript{$\ddagger$} Google Inc., Mountain View, CA 94043, USA\\
\par\bigskip 
}

\begin{abstract}
% \wjs{Use the first paragraph to highlight the importance of privacy in multiple testing?}\red{{\em Hypothesis testing} refers to the use of a statistical procedure to assess the validity of a putative scientific discovery, for example, that a given treatment for a condition is more effective than a placebo.  The difficulty in hypothesis testing is to distinguish random effects due to sampling from true nulls.  The problem is exacerbated when the number of hypotheses is large, and the focus of a large subfield in statistics is the relaxation of the goal of preventing all false discoveries to instead controlling the {\em rate} of false discovery.  }

We provide the first differentially private algorithms for controlling the false discovery rate (FDR) in multiple hypothesis testing, with essentially no loss in power under certain conditions.  Our general approach is to adapt a well-known variant of the Benjamini-Hochberg procedure (\BHq{}), making each step differentially private. This destroys the classical proof of FDR control. To prove FDR control of our method we
\begin{enumerate}
\item Develop a new proof of the original (non-private) BHq algorithm and its robust variants -- a proof requiring only the assumption that the true null test statistics are independent, allowing for arbitrary correlations between the true nulls and false nulls. This assumption is fairly weak compared to those previously shown in the vast literature on this topic, and explains in part the empirical robustness of BHq.
\item
Relate the FDR control properties of the differentially private version to the control properties of the non-private version.
\end{enumerate}
We also present a low-distortion ``one-shot'' differentially private primitive for ``top $k$'' problems, e.g., ``Which are the $k$ most popular hobbies?'' (which we apply to: ``Which hypotheses have the $k$ most significant $p$-values?''), and use it to get a faster privacy-preserving instantiation of our general approach at little cost in accuracy.  The proof of privacy for the one-shot top~$k$ algorithm introduces a new technique of independent interest. 
\end{abstract}

\section{Introduction}
\label{sec:intro}

{\em Hypothesis testing} is the use of a statistical procedure to assess the validity of a given hypothesis, for example, that a given treatment for a condition is more effective than a placebo. The traditional approach to hypothesis testing is to (1) formulate a pair of {\em null} (the treatment and the placebo are equally effective) and {\em alternative} (the treatment is more effective than the placebo) hypotheses; (2) devise a test statistic for the null hypothesis, (3) collect data; (4) compute the {\em $p$-value} (the probability of observing an effect as or more extreme than the observed effect, were the null hypothesis true; smaller $p$-values are ``better evidence'' for rejecting the null); and (5) compare the $p$-value to a standard threshold $\alpha$ to determine whether to {\em accept} the null hypothesis (conclude there is no interesting effect) or to {\em reject} the null hypothesis in favor of the alternative hypothesis.

In the {\em multiple hypothesis testing} problem, also known as the {\em multiple comparisons} problem, $p$-values are computed for multiple hypotheses, leading to the problem of false discovery: since the $p$-values are typically defined to be uniform in $(0,1)$ for true nulls, 
%where the probability is over the sampling,
by definition we expect an $\alpha$ fraction of the true nulls to have $p$-values bounded above by~$\alpha$.  

Multiple hypothesis testing is an enormous problem in practice; for example, in a single genome-wide association study a million SNPs\footnote{A {\em single nucleotide polymorphism}, or SNP, is a location in the DNA in which there is variation among individuals.} may be tested for an association with a given condition.  Accordingly, there is a vast literature on the problem of controlling the {\em false discovery rate} (FDR), which, roughly speaking, is the expected fraction of erroneously rejected hypotheses among all the rejected hypotheses, where the expectation is over the choice of the data and any randomness in the algorithm (see below for a formal definition).  

The seminal work of Benjamini and Hochberg~\cite{BenjaminiH95} and their beautiful ``BHq'' procedure (Algorithm~\ref{algo:originalbh} below) is our starting point. Assuming independence or certain positive correlations of $p$-values \cite{prds}, this procedure controls FDR at any given nominal level $q$.  An extensive literature explores the control capabilities and power of this procedure and its many variants\footnote{The power of a procedure is its ability to recognize ``false nulls'' (what a lay person may describe as ``true positives'').}.

%, and differentially private algorithms occupy these interstices. %\footnote{Note that the Law recognizes a privacy/accuracy tradeoff, and differentially private algorithms provide instantiations.  As a consequence, when we talk about privacy, we often write ``privacy/accuracy.''}.  
%
{\em Differential privacy}~\cite{DworkMNS06} is a definition of privacy tailored to statistical data analysis. The goal in a differentially private algorithm is to hide the presence
or absence of any individual or small group of individuals, the
intuition being that an adversary unable to tell whether or not a
given individual is even a member of the dataset surely cannot glean
information specific to this individual.  To this end, a pair of
databases $\dbx,\dby$ are said to be {\em adjacent} (or {\em
neighbors}) if they differ in the data of just one individual.
Then a differentially private algorithm
(Definition~\ref{def:dp}) ensures that the behavior on adjacent
databases is statistically close, so it is difficult to infer if any particular individual is in the database from the output of the algorithm and any other side information.

%, that
%is, one is contained in the other and the larger database has the data
%of just one more person.
\remove{
$(\eps,0)$-differentially private algorithms
(Definition~\ref{def:dp}) ensure that the behavior on adjacent
databases is statistically close: the ratio of the probabilities of
any given event on neighboring databases $\dbx,\dby$ is bounded
by~$\ep\eps$.  Very roughly speaking, the more the data of one
individual can change the output of an algorithm (in the worst case),
the more difficult it is to hide the difference between adjacent
databases.}

\paragraph{Contributions.} 
We provide the first differentially private algorithms for controlling the FDR in multiple hypothesis testing.  The problem is difficult because the data of a single individual can affect the $p$-values of all hypotheses simultaneously. %In addition, the $p$-value of single element can change the outcome of BHq procedure dramatically. 
Our general approach is to adapt the Benjamini-Hochberg step-down procedure, making each step differentially private. This destroys the classical proof of FDR control. To prove FDR control of our method we 
\begin{enumerate}
\item Develop a new proof of the original (non-private) BHq algorithm and its robust variants -- a proof requiring a fairly weak assumption compared to those previously shown in the vast literature on this topic -- and
\item
Relate the FDR control properties of the differentially private version to the control properties of the non-private version. Power is also argued in relation to the non-private version.
\end{enumerate}

Central to BHq and its variants is the procedure for reporting the experiments with $k$ most significant $p$-values, or known as the top-$k$ problem. To our knowledge, the most accurate approximately private top-$k$ algorithm is a ``peeling'' procedure, in which one runs a differentially private maximum procedure $k$ times, tuning the ``inaccuracy'' to roughly $\sqrt{k}/\eps$.  Here we present a low-distortion ``one-shot'' differentially private primitive for ``top $k$'' with the same dependence on~$k$, and use it to get a faster privacy-preserving instantiation of our general approach.\footnote{For pure differential privacy, previous work (e.g.~\cite{bhaskar2010kdd}) provides a way to avoid peeling by adding Laplace noise of scale $O(k)$. Conceivably, it is much more challenging to get the dependence on $k$ down to $\sqrt{k}$ for approximate differential privacy, which is what we show in this paper.} The proof of differential privacy for our one-shot top~$k$ algorithm introduces a new technique of independent interest.

Finally, traditional methods of computation of a $p$-value are not typically differentially private.  Our privacy-preserving algorithms for control of the FDR require the $p$-value computation to satisfy a technical condition.  We provide such a method for computing $p$-values, ensuring ``end to end'' differential privacy from the raw data to the set of rejected hypotheses.
%\linote{Is this referring to the lower bound of $p$-value?}
%\cnote{yes, Li.  That is what I had in mind.}
%\linote{OK}

\remove{
{\em ``One-Shot'':} A generalization of the \RNMa{} procedure described above that allows us to retrieve the indices of the approximately $k$ largest computations while paying an accurate price of only $\tilde{O}(\sqrt{k})$ (rather than the more obvious price of $O(k)$). Previous algorithms achieving comparable bounds required $k$ iterations, during each of which the top (remaining) element was identified and removed from further consideration.  Our algorithm operates in one shot, just like \RNMa .  \cnote{What about the tree?}\linote{I mentioned it in footnote 3 and also in the description of one-shot.}
The proof of privacy introduces a new technique of independent interest. 
}

\subsection{Description of Our Approach}
\paragraph{The \BHq{} procedure.}  Suppose we are simultaneously testing $m$ null hypotheses $H_1, \ldots, H_m$ and $p_1, \ldots, p_m$ are their corresponding $p$-values. Denote by $R$ the number of rejections (discoveries) made by any procedure and $V$ the number of true null hypotheses that are falsely rejected (false
discoveries). The FDR is defined as
\[
\fdr := \E \left[ \frac{V}{\max\{R, 1\}} \right] \,.
\]
%%where the semi-colon notation amounts to always interpreting $V/R$ as zero if $V = 0$ (note that $R = 0$ implies $V=0$).

Algorithm~\ref{algo:originalbh} presents the original \BHq procedure for controlling FDR at~$q$. The thresholds $\bhc_j = qj/m$ for $1 \le j \le m$, are known as the {\em BHq critical values}.  The following intuition may demystify the \BHq{} algorithm. In the case that all $m$ null hypotheses are true and their $p$-values are \iid uniform on $(0, 1)$, then we expect, approximately, a $qj/m$ fraction of the $p$-values to
lie in the interval $[0, qj/m]$.  If instead there are at least $j$ many $p$-values in this interval (this is precisely what the condition $p_{(j)} \le qj/m$ says), then there are ``too many'' $p$-values in $[0,p_{(j)}]$ for all of them to correspond to null hypotheses: We have $j$ such $p$-values, and should attribute no more than $jq$ of these to the true nulls; so if we reject all these $j$ hypotheses we would expect that at most a $q$ fraction correspond to true nulls.

\begin{algorithm}[ht]
\caption{Step-up BHq procedure}\label{algo:originalbh}
\begin{algorithmic}[1]
\REQUIRE Level $0<q<1$ and $p$-values $p_1, \ldots, p_m$ of hypotheses, respectively, $H_1, \ldots, H_m$
\ENSURE a set of rejected hypotheses
\STATE sort the $p$-values in increasing order: $p_{(1)}\leq p_{(2)}\leq \cdots\leq p_{(m)}$
\FOR{$j=m$ to $1$}
\IF {$p_{(j)} >  qj/m$}
\STATE continue
\ELSE
\STATE reject $H_{(1)}, \ldots, H_{(j)}$ and halt
\ENDIF
\ENDFOR
\end{algorithmic}
\end{algorithm}
% \IF {$p_{(j)}\leq \alpha_j$}
%   \STATE reject all remaining hypotheses (include itself) and halt
% \ELSE
%   \STATE remove $(j)$ from the set of hypotheses under consideration

\paragraph{Making \BHq{} Private.} There is an extensive literature on, and burgeoning practice of, differential privacy.  For now, we need only two facts: (1) Differential privacy is closed under composition, permitting us to bound the cumulative privacy loss over multiple differentially private computations.  This permits us to build complex differentially private algorithms from simple differentially private primitives; and (2) we will make use of the well known \RNMa{} (respectively, \RNMi ) primitive, in which appropriately distributed fresh random noise is added to the result of each of $m$ computations, and the index of the computation yielding the maximum (respectively, minimum) noisy value is returned. By returning only one index the procedure allows us to pay an accuracy price for a single computation, rather than for all~$m$.\footnote{The variance of the noise distribution depends on the maximum amount that any single datum can (additively) change the outcome of a computation and the inverse of the privacy price one is willing to pay.}

A natural approach to making the \BHq{} procedure differentially private is by repeated use of the \RNMa{} primitive: Starting with $j=m$ and decreasing: use \RNMa to find the hypothesis $H_{j}$ with the (approximately) largest $p$-value; estimate that $p$-value and, if the estimate is above (an appropriately more conservative critical value) $\alpha_j < \bhc_j$, accept $H_j$, remove $H_j$ from consideration, and repeat. Once an $H_j$ is found whose $p$-value is below the threshold, reject all the remaining hypotheses.  The principal difficulty with this approach is that every iteration of the algorithm incurs a privacy loss which can only be mitigated by increasing the magnitude of the noise used by \RNMa. Since each iteration corresponds to the acceptance of a null hypothesis, the procedure is paying in privacy/accuracy precisely for the null hypotheses accepted, which are by definition not the ``interesting'' ones.  
%For the very large values of $m$ routinely now under consideration, this yields a prohibitive accuracy loss.  

If, instead of starting with the largest $p$-value and considering the values in decreasing order, we were to start with the smallest $p$-value and consider the values in increasing order, rejecting hypotheses one by one until we find a $j \in [m]$ such that $p_{(j)} > \alpha_j$, the ``demystification'' intuition still applies. This widely-studied variation is called the Benjamini-Hochberg {\em step-down} procedure, and hereinafter we call it the step-down BHq (the original BHq is referred to as the step-up BHq in contrast).  Their definitions reveal that the step-down variant shall be more conservative than the other. However, as shown in \cite{stepdown}, the step-down BHq can assume less stringent critical values than the BHq critical values while still controls FDR, often allowing more discoveries than the step-up \BHq. In a different direction, \cite{abdj} establishes that under a weak assumption on the sparsity the difference between these two BHq's is asymptotically negligible and both procedures are optimal for Gaussian sequence estimation over a wide range of sparsity classes. This remarkable result implies that, as a pure testing criterion, the FDR concept is fundamentally correct for estimation problems.

%%This remarkable result implies that the FDR criterion is fundamentally correct in connecting testing and estimation problems.
\begin{algorithm}[!htp]
\caption{Step-down BHq procedure}\label{algo:smallbh}
\begin{algorithmic}[1]
\REQUIRE $0<q<1$ and $p$-values $p_1, \ldots, p_m$ of hypotheses, respectively, $H_1, \ldots, H_m$
\ENSURE a set of rejected hypotheses
\STATE sort the $p$-values in increasing order: $p_{(1)}\leq p_{(2)} \le \cdots \leq p_{(m)}$
\FOR{$j=1$ to $m$}
\IF {$p_{(j)} \le qj/m$}
  \STATE reject $H_{(j)}$
\ELSE
  \STATE halt %return
\ENDIF
\ENDFOR
\end{algorithmic}
\end{algorithm}

% %\cnote{removed the adaptive step-down FDR procedure from the introduction}
% \remove{
% \begin{algorithm}[ht]
% \caption{Adaptive step-down FDR procedure}\label{algo:smallbh}
% \begin{algorithmic}[1]
% \REQUIRE $0<q<1$, and $m$ $p$-values $p_1, \cdots, p_m$
% \ENSURE the set of rejected hypotheses
% \STATE sort the $p$-values in increasing order, obtaining $p_{(1)}\leq p_{(2)}\leq\ldots\leq p_{(m)}$;
% \FOR{$j=1$ to $m$}
% \IF {$p_{(j)}\leq \alpha_j$}
%   \STATE reject $(j)$;
% \ELSE
%   \STATE return
% \ENDIF
% \ENDFOR
% \end{algorithmic}
% \end{algorithm}
% }

If we make the natural modifications to the step-down BHq (using \RNMi now,
instead of \RNMa), then we pay a privacy cost only for nulls rejected
in favor of the corresponding alternative hypotheses, which by
definition are the ``interesting'' ones.  Since the driving
application of \BHq{} is to select promising directions for future
investigation that have a decent chance of panning out, we can view
its outcome as advice for allocating resources. Thus, a procedure that
finds a relatively small number~$k$ of high-quality hypotheses, still achieving FDR control, may be as useful as a procedure that finds a much larger set.

% \remove{
% %%% \linote{This paragraph seems repetitive.}
% With this in mind, Algorithm~\peel takes as a parameter a maximum
% number $k$ of rejections, together with a privacy parameter~$\eps$,
% and repeatedly runs \RNMi{} on appropriately noised $p$-values,
% ``peeling off'' approximate noisy $p$-values, comparing them to the
% appropriate threshold, halting if the value is too large, and
% otherwise rejecting the corresponding null and iterating on the
% remaining hypotheses.  Algorithm~\Fast{} uses the one-shot primitive
% to select the hypotheses with the $k$ (approximately) smallest
% $p$-values.  It then adds fresh random noise to the $p$-values for the
% selected hypotheses and runs \BHsmall{} on these $k$ noisy
% $p$-values.}

\paragraph{Proving FDR Control.}
A key property of the two BHq procedures is as follows: if $R$ rejections are made, the maximum $p$-value of all the rejected hypotheses is bounded above by $qR/m$. This motivates us to formulate the definition below.
\begin{definition}\label{def:adaptive}
Given any critical values $\{ \alpha_j\}_{j=1}^m$, a multiple testing procedure is said to be {\em adaptive} to $\{\alpha_j\}_{j=1}^m$, if either it rejects none or the rejected $p$-values are all bounded above by $\alpha_R$, where the number of rejections $R$ can be arbitrary.
\end{definition}

In this paper, we often work with the sequence $\{ \alpha_j\}_{j=1}^m$ set to be the BHq critical values $\{\bhc_j\}_{j=1}^m$. As motivating examples, both BHq's in addition obey the property that the rejected hypotheses are contiguous in sorted order; for example, it is impossible that the hypotheses with the smallest and third smallest
$p$-values are rejected while the hypothesis with the second smallest $p$-value is accepted. Nevertheless, in general an adaptive procedure does not necessarily have property as it may skip some smallest $p$-values. For protecting differential privacy, this relaxation is desirable because differentially private algorithms may not reject consecutive minimum $p$-values due to artificial noise introduced. Note that this perfectly matches our ``demystification'' intuition. 

Remarkably, an adaptive procedure with respect to the BHq critical values still controls the FDR. This result also applies to a generalized FDR defined for any positive integer $k$ \cite{kfdr,kfdrfollow}:
\[
\fdr_k := \E \left[ \frac{V}{R}; V \ge k \right].
\]
Built on top of the original FDR, this generalization does not penalize if fewer than $k$ false discoveries are made. Note that it reduces to the original FDR in the case of $k = 1$.
\begin{theorem}\label{thm:robustfdr}
Assume that the test statistics corresponding to the true null
hypotheses are jointly independent. Then any procedure adaptive to the \BHq
critical values $\bhc_j = qj/m$ obeys
\begin{equation}\label{eq:fdr2}
\begin{aligned}
&\fdr \leq q\log(1/q) + Cq\,,\\
&\fdr_2 \leq Cq\,,\\
&\fdr_k \le \left( 1 + 2/\sqrt{qk} \right) q\,,
\end{aligned}
\end{equation}
where $C < 3$ is a universal constant.

\end{theorem}

%\wjnote{remove?: To be in line with the literature, let us call null $p$-values \textit{true null $p$-values}, and alternative $p$-values \textit{false null $p$-values}.} 
%By definition, if the null is true, then the $p$-value computed by the test statistics is stochastically smaller than or equal to uniform.}

One novelty of Theorem~\ref{thm:robustfdr} lies in the absence of
any assumptions about the relationship between the true null test
statistics and false null test statistics. In the literature, independence, or some (very stringent) kind of positive dependence \cite{prds} between these two groups of test statistics, is necessary for provable FDR control. In a different line, Theorem 1.3 of \cite{prds} controls FDR without any assumptions by using the (very stringent) critical values ${qj}/{(m\sum_{i=1}^m\frac1{i})} \approx {qj}/{(m\log m)}, 1 \le j \le m$, effectively paying a factor of $\log m$, whereas we
pay only a constant factor. This  simple independence
assumption within the true nulls shall also capture more real life scenarios.
As we will see, the additive $q \log (1/q)$ term for the original
FDR, i.e. $\fdr_1$, is unavoidable given so few assumptions. Surprisingly, $\fdr_2$ no longer has this dependency, and $\fdr_k$ even approaches $q$ as $k$ grows.

\remove{
To prove Theorem~\ref{thm:robustfdr} (Sec~\ref{sec:robust}), we
first reduce this problem to studying the key quantity $\E\left[ \max_{2 \le j \le m}
j/\lceil m U_{(j)}/q\rceil \right]$, where $U_{(j)}$ is
the $j$th smallest statistic of $m$ \iid uniform random variables on $(0, 1)$.  If we
express these order statistics in terms of sums of exponentially
distributed random variables, then the key quantity has a nice
martingale property, and we use known techniques to bound its
expectation. }
%To prove the asympototic bound~(\ref{Ck}), we actually
%show a stronger bound by considering $\mathrm{FDR}^k \defeq \E
%(\frac{V}{R}; R \ge k )\geq \fdr_k,$ and showing that $\fdr^k \leq
%(1+2/\sqrt{qk})q$.
%\cnote{I suppressed all mention of $\mathrm{FDR}^k$ from the introduction.}
%\linote{OK.}

Theorem~\ref{thm:robustfdr} is the key to proving FDR control even if
the procedure is given ``noisy'' versions of the $p$-values, as
happens in our differentially private algorithms.  To determine how to
add noise to ensure privacy, we study the {\em sensitivity} of a
$p$-value, a measure of how much the $p$-value can change between
adjacent datasets
(Section~\ref{sec:prelim:sensitivity})\footnote{Recall that adjacent
datasets differ in the data of just one person.}. For standard
statistical tests this change is best measured multiplicatively,
rather than additively, as is typically studied in differential
privacy.  Exploiting multiplicative sensitivity is helpful when the
values involved are very small. Since we are interested in the regime
where $m$, the number of hypotheses, is much larger than the number of
discoveries, the $p$-values we are interested in are quite small: the
$j$th BHq critical value is only $qj/m$. We remark that adapting algorithms
such as \RNMi{} and the one-shot top~$k$ to incorporate multiplicative
sensitivity can be achieved by working with the logarithms of the
values involved.

Informally, we say a $p$-value is $\eta$-multiplicatively sensitive truncated at $\nu$ if
for any two neighboring databases $D$ and $D'$, the $p$-value computed
on them are within multiplicative factor of $\ep\eta$ of each other, unless they are smaller than $\nu$ (See Definition~\ref{def:multi_sensi}). The justification of the multiplicative sensitivity is provided by the definition of some standard $p$-values under independent bounded statistics.  Indeed as we show in Section~\ref{sec:prelim:sensitivity}, such $p$-values are $\wO(1/\sqrt{n})$ multiplicatively sensitive\footnote{We use $\wO$ to hide polynomial factors in $\log (m/\delta)/\e$.}. With these results, we can show that our algorithm controls FDR under the bound given in Theorem~\ref{thm:robustfdr} while having the comparable power to the \BHsmall$(q)$.
\begin{theorem}[informal]\label{thm:PrivateFDRInformal}
Suppose $\sqrt{k}\eta/\epsilon = \tilde o(1)$. Given any nominal level $0 < q <1$, the differentially private \BHq{} algorithms controls the FDR at level $q + o(1)$. 
\end{theorem}

\begin{theorem} [informal]\label{thm:PrivatePowerInformal}
Again, suppose $\sqrt{k}\eta/\epsilon = \tilde o(1)$. With probability $1 - o(1)$, the differentially private \BHq{} algorithms with target nominal level $q$ makes at least as many discoveries as the \BHq{} step-down procedure truncated at $k$ with nominal level $(1-o(1))q$.

% is at least as powerfulcontrols the FDR at level $q + o(1)$. For independent bounded statistics, if $k \ll n/ \log^{O(1)} m$, then
% with probability $1-o(1)$, our differentially private algorithms
% reject at least the same number of hypotheses as \BHsmall(q) while
% maintaining FDR control with the bound given in Theorem~\ref{thm:robustfdr} provided $q$ is replaced by $(1+o(1))q$.
\end{theorem}

% \remove{
% Our algorithms for FDR control assume certain properties about the sensitivity of the $p$-values.  We show how to compute these values accordingly (Section~\ref{sec:defs}).

% The next two theorems summarize our results for privacy-preserving FDR control
% see Sections~\ref{privfdr}.
% \cnote{We need proofs of formal versions of Theorems \ref{thm:PrivateFDRInformal} and \ref{thm:PrivatePowerInformal} in the appropriate subsections.  Fill in the subsection references in the previous sentence.}
% \noindent
% The power of the private FDR algorithm depends on the sensitivity of
% the $p$-values.  Combining with our analysis of the sensitivity of
% $p$-values under some standard settings, we can show that the private FDR
% only lose a little power. In (Section~\ref{}),\cnote{Fill in section
% name.} we show that
% In Section~\ref{sec:mono} we XYZ \cnote{finish or delete}
% }

\paragraph{The One-Shot Top~$k$ Mechanism.}  Consider a database of $n$
binary strings $d_1,d_2,\ldots,d_n$. Each $d_i$ corresponds to an
individual user and has length $m$. Let $x_j=\sum_i d_{ij}$ be the
total sum of the the $j$th column.  In the reporting top-$k$
problem\footnote{In our paper, it is more convenient to consider the
minimum-$k$ (or bottom-$k$) elements. But we still call it the top-$k$
problem following the convention.}, we wish to report, privately, some
locations $j_1,\ldots, j_k$ such that $x_{j_\ell}$ is close to the
$\ell$th smallest element as possible.  The peeling mechanism reports
and removes the minimum noisy count and repeats on the remaining
elements, each time adding fresh noise.  Such a mechanism is
$(\e,\delta)$-differentially private if we add
$\lap(\sqrt{k\log(1/\delta)}/\e)$ noise each time. 

In contrast, in the one-shot top $k$ mechanism, we add
$\wO(\sqrt{k})$ noise to each value and then report the $k$ locations
with the minimum noisy values.\footnote{For technical reasons, if we
want the values of the computations in addition to their indices we
only know how to prove privacy if we add fresh random noise before
releasing these values.}  Compared to the peeling mechanism, the
one-shot mechanism is appealingly simple and much more efficient. But
it is surprisingly challenging to prove its privacy.  Here we give
some intuition. 

When there are large gaps between $x_i$'s, the change of one
individual can only change the value of each $x_i$ by at most $1$, so
result of the one-shot algorithm is stable. Hence the privacy is
easily guaranteed.  In the more difficult case, when there are many
similar values (think of the case when all the values are equal), the
true top~$k$ set can be quite sensitive to the change of input values.
But in the one-shot algorithm we add independent symmetric noise,
centered at zero, to these values.  Speaking intuitively, this yields
an (almost) equal chance that on two adjacent input values a noisy
value will ``go up'' or ``go down,'' leading to cancellation of
certain first order terms in the (logarithms of) the probabilities of
events and hence a tight control between their ratio.
% and leads to the dominance of the second order term.

To capture this intuition, we consider the ``bad'' events, which have
large probability bias between two neighboring inputs. Those bad
events can be shown to happen when the sum of some dependent random
variables deviates from its mean.  The dependencies among the random
variables prevents us from applying a concentration bound directly.
To deal with this difficulty, we first partition the event spaces to
remove the dependence.  Then we apply a coupling technique to pair up
the partitions for the two neighboring inputs. For each pair we apply
a concentration bound to bound the probability of bad events. The
technique appears to be quite general and might be useful in other
settings.

\section{Preliminaries on Differential Privacy}
\label{sec:defs}
We revisit some basic concepts in differential privacy.
\begin{definition}
Data sets $\dbax,\dbby$ are said to be {\em neighbors}, or {\em adjacent}, if one is obtained by removing or adding a single data item.
\end{definition}
%%We use $\widetilde{O}$ to hide polynomial factors in $\log m$, $1/\e$, and $\log(1/\delta)$.  We fix $0\leq q\leq 1$ throughout the writeup.

Differential privacy, now sometimes called {\em pure differential privacy}, was defined and first constructed in~\cite{DworkMNS06}.  The relaxation defined next is sometimes referred to as {\em approximate differential privacy}.  
\begin{definition}[Differential privacy~\cite{DworkMNS06,DworkKMMN06}] 
\label{def:dp}
A randomized mechanism $M$ is $(\e,\delta)$-differentially private if
for all adjacent $\dbx,\dby$, and for any event $S$: $\prob_D[S]\leq
\ep{\e} \prob_{D'}[S] + \delta$. 
Pure differential privacy is the special case of approximate differential privacy in which~$\delta=0$.
\end{definition}

Denote by $\lap(\lambda)$ the Laplace distribution with scale $\lambda > 0$, whose probability density function reads
\[
f_{\lap}(z) = \frac1{2\lambda} \ep{-|z|/\lambda}.
\]

% \[G_{\lambda}(z) = \left\{
% \begin{array}{ll}
% \frac{1}{2}\ep{z/\lambda} & \mbox{if $z<0$,}\\
% 1-\frac{1}{2}\ep{-z/\lambda} &\mbox{if $z>0$.}
% \end{array}\right.\]

\begin{definition}[Sensitivity of a function]
Let $f$ be a function mapping databases to $\R^k$.  The {\em sensitivity of $f$}, denoted $\Delta f$, is the maximum over all pairs $\dba,\dbb$ of adjacent datasets of $\|f(\dba)-f(\dbb)\|$.
\end{definition}

\remove{Differential privacy is closed under composition, even in the {\em adaptive} case, in which the $i$th computation is chosen as a function of the (differentially private3) outputs of the first $i-1$ computations.}

We will make heavy use of the following two lemmas on differential privacy.
\begin{lemma}[Laplace Mechanism~\cite{DworkMNS06}]\label{lm:comp}
Let $f$ be a function mapping databases to $\R^k$.  The mechanism that, on database~$\dba$, adds independent random draws from $\lap((\Delta f)/\eps)$ to each of the $k$ components of $f(\dba)$, is $(\eps,0)$-differentially private.
\end{lemma}

\begin{lemma}[Advanced Composition~\cite{DworkRV10}]\label{lm:advancecomp}
For all $\eps, \delta, \delta' \ge 0$, the class of $(\eps, \delta')$-differentially private mechanisms satisfies 
$( \sqrt{2k \ln (1/\delta)}\eps +  k\eps(e^\eps -1)/2, k\delta' + \delta)$-differential privacy under $k$-fold adaptive composition.
\end{lemma}

\subsection{Reporting top-$k$ elements}

Consider the problem of privately reporting the minimum $k$ locations
of $m$ values $x_1, \ldots, x_m$.  Here two input values $(x_1,
\ldots, x_m)$ and $(x'_1, \ldots, x'_m)$ are called adjacent if 
$|x_i-x'_i|\leq 1$ for all $1\leq i\leq m$.  When $k=1$, this is
solved by Algorithm~\ref{algo:rnm} stated below.
\begin{algorithm}[ht]
\caption{Report Noisy Min}
\label{algo:rnm}
\begin{algorithmic}[1]
\REQUIRE $m$ values $x_1, \cdots, x_m$
\ENSURE $j$
\FOR {$i=1$ to $m$}
  \STATE set $y_i = x_i+g_i$ where $g_i$ is sampled
  i.i.d. from $\lap(2/\e)$
\ENDFOR
\STATE return $(i_j = \argmin_i ~ y_i, x_{i_j} + g')$, where $g'$ is fresh random noise sampled from $\lap(2/\e)$
\end{algorithmic}
\end{algorithm}

A standard result in differential privacy 
\begin{lemma}
Algorithm~\ref{algo:rnm} is $(\e,0)$-differentially private.
\end{lemma}

Notably, directly reporting $y_j$ shall violate pure privacy.  That is why we need to add fresh random noise to
$x_j$.  For the top-$k$ problem, one can apply the above process $k$
times, each time removing the output element and applying the
algorithm to the remaining elements, hence called the \emph{peeling
mechanism}. By composition theorem, it is immediate such a mechanism
is $(\e,0)$-differentially private if we add noise $\lap(k/\e)$ at
each step. Or one can add $\lap(\sqrt{k\log(1/\delta)}/\e)$ noise to
get $(\e,\delta)$-differential privacy by applying the advanced composition
theorem~\cite{dwork2014algorithmic}. 

\begin{theorem}\label{thm:peeling}
The peeling mechanism is $(\e,\delta)$-differentially private. Assume $k = o(m)$, then with probability $1-o(1)$, for every $1\leq j\leq k$,
\[
x_{i_j}-x_{(j)} \leq (1 + o(1))\frac{\sqrt{k\log(1/\delta)} \log m} {\e} \,.
\]
\end{theorem}
One naturally asks if we can avoid peeling by adding the noise $\lap(\wO(\sqrt{k})$ once and then reporting the $k$-minimum elements.  Such a \emph{one-shot} mechanism would be much more efficient and more natural. One major contribution of this paper is to show indeed the one-shot mechanism is private.

%%% Local Variables:
%%% mode: latex
%%% TeX-master: "paper"
%%% End:

\section{FDR Control of Adaptive Procedures}\label{sec:robust}

We now study the FDR control of the adaptive procedures introduced in Definition~\ref{def:adaptive} and then prove Theorem~\ref{thm:robustfdr}. The class of adaptive procedures includes both the step-up and step-down BHq's. This serves as the first step in proving FDR control properties of our differentially private algorithm.

The ``natural'' way of making the step-down BHq differentially private is to add sufficient noise to each $p$-value to ensure privacy and then run this procedure on the noisy $p$-values. But working with noisy $p$-values completely destroys some crucial properties for proving FDR control of step-down BHq. In particular, the $j$th most significant noisy $p$-value may not necessarily correspond to the $j$th most significant true $p$-value. As a result, it is possible that this (in fact any) procedure may be comparing this noisy $p$-value to the ``wrong'' critical values. In addition, some hypotheses with small $p$-values may even not be rejected due to the noise, so unlike in the non-private \BHq and its variants, there may be gaps in the sequence of $p$-values of the rejected hypotheses.

In the following, we show that even under the above unfavorable conditions, the FDR is still controlled for any procedure adaptive to the BHq critical values $\bhc_j$. This provides theoretical justification for the robustness of BHq observed in a wide range of empirical studies.  While our goal is to show that the step-down BHq still controls the FDR when the $p$-values of the input are randomly perturbed, it is more convenient and general to
consider the procedures which relax the condition of \BHsmall but on unperturbed (true) $p$-values. 

To this end, the class of \textit{adaptive} procedures, including both \BHq and \BHsmall as special examples, is characterized in Definition~\ref{def:adaptive}. Compared with \BHsmall, an adaptive procedure only requires all the rejections below a critical value that depends on the number of rejections, but not that each $p$-value is below the respective critical value. Further, it allows to reject only a subset of $p$-values below some threshold. Surprisingly, even so an adaptive procedure still controls the FDR as stated in Theorem~\ref{thm:robustfdr}. An novel element in Theorem~\ref{thm:robustfdr} is that it only assumes independence within the true null test statistics. Compared with other assumptions in the literature, this new assumption is particularly
simple and has the potential to capture many real life examples. 

In the original proof of FDR control by \BHq \cite{BenjaminiH95}, it assumes (i) the true null
statistics are independent, and (ii) the false null statistics are
independent from the true null statistics. Later,
\cite{prds} proposed a slightly weaker sufficient condition for FDR
control called positive regression dependency on subsets, which,
roughly, amounts to saying that every null statistic is positively
dependent on all true null statistics. In sharp contrast, our theorem
makes no assumption about the correlation structure between the false
null statistics and true null statistics, at a mere cost of a constant
multiplying the nominal level. To see an example that only satisfies
our assumption, take \iid uniform random variables on $(0, 1)$ as the
true null $p$-values, and let all the false null $p$-values be 1 minus
the median of the true null $p$-values.

We consider the bounds on $\fdr,\fdr_2$ and on $\fdr_k$ separately. Here we outline the proof for $\fdr$ and $\fdr_2$ and leave $\fdr_k$ case in the supplemented full version. In the following we prove Theorem~\ref{thm:robustfdr} in two
parts. First we show the claimed bounds for $\fdr$ and $\fdr_2$.  For
$\fdr_k$, we actually consider a related quantity 
\[
\fdr^k := \E \left[\frac{V}{R}; R \ge k \right] \,.
\]
Since $V\leq R$, clearly we have that
$\fdr_k \leq \fdr^k$. We will show (\ref{eq:fdrk}) holds even for
$\fdr^k$.

\subsection{Controlling $\fdr$ and $\fdr_2$}
%\cnote{said ``weakly'' dominated}

We will prove the result by constructing the most ``adversarial'' set
of $p$-values. Imagine the following game for a powerful adversary $A$ who are informed of all the $m_1$ false null hypotheses and can even set $p$-values for them. The remaining $p$-values, which are all from the true nulls, are then drawn from \iid $U(0, 1)$. Then $A$ can pick out a subset $S$ of $p$-values with the only
requirement that those $p$-values are upper bounded by $\alpha_{|S|}$ for
the critical values $\alpha_j = jq/m$. $A$'s payoff is then the ratio of the
true nulls (i.e. FDR) in $S$.  The expected payoff of $A$ would be the
upper bound on $\fdr$ for any adaptive procedure. If we require $A$
only receive payoff when he includes at least $k$ true nulls in $S$,
then the corresponding payoff is an upper bound of $\fdr_k$.

First how should $A$ set the $p$-values of alternatives? $0$! because
this way $A$ can include any number of alternatives in $S$ to push up
the size of $S$ so raise up the critical values but without wasting any
``space'' for including more true nulls. With this we can reduce
bounding $\fdr$ to bounding the expected value of a random variable. 

We now present the rigorous argument below.  Denote by
$\mathscr{N}_0$, with cardinality $m_0$, the set of true null
hypotheses, and $\mathscr{N}_1$, with cardinality $m_1$, the set of
false nulls.  Define $\fdp_k = V/(R\vee 1)$ for $V\geq k$ and $0$
otherwise.  Given a realization of $p_1, \ldots, p_m$, we would like
to obtain a tight upper bound on $V/(R \vee 1)$ with the constraint
that the maximum of the rejected $p$-values is no larger than
$\alpha_R$. With this in mind, call $(p_{i_1}, p_{i_2},\ldots,
p_{i_R})$ the rejected $p$-values, among which $V$ of them are from
the $m_0$ many true null $p$-values. Hence, denoting by $p^0_1,
\ldots, p^0_{m_0}$ the true null $p$-values, we see 
\[
p^0_{(V)} \le \max_{1 \le j \le R} p_{i_j} \le \alpha_R.
\]
Taking the
\BHq critical value $\alpha_R = qR/m$ and rearranging this inequality yield $R
\ge \lceil m p^0_{(V)}/q \rceil$, which also makes use of the
additional information that $R$ is an integer. As a consequence, we
get $V/(R \vee 1) \le V/\lceil m p^0_{(V)}/q \rceil$.
Hence, it follows that
\begin{equation}\label{eq:fdp_main}
\fdp_k \le \max_{k \le j \le m_0} \frac{j}{\ceil{m p^0_{(j)}/q}} \,.
\end{equation}

We pause here to point out that this upper bound is tight for the
class of procedures adaptive to the \BHq critical values. Let $j^\star$ be the
index where the maximization of \eqref{eq:fdp_main} is attained. In
constructing the least ``favorable'' set of $p$-values and the most
``adversarial'' procedure, we set $p_j = 0$ whenever $j \in
\mathscr{N}_1$, and let the procedure reject the $j^\star$ smallest
true null $p$-values and $\ceil{m p^0_{(j^\star)}/q} - j^\star$ false
null $p$-values (which are all zero). It is easy to verify the
adaptivity of this case.

Recognizing that the true null $p$-values are independent and
stochastically no smaller than uniform on $(0, 1)$, we may assume the
true null $p$-values are \iid uniform on $(0, 1)$ in taking
expectations of both sides of \eqref{eq:fdp_main}. In addition, it is
easy to see that increasing $m_0$ to $m$ only makes this inequality
more likely to hold. Hence, we have
\begin{equation}\label{eq:fdr_k_upperbound}
\fdr_k \le \E \left[ \max_{k \le j \le m}\frac{j}{\ceil{m U_{(j)}/q}} \right] \,,
\end{equation}
where, as earlier, $U_{(1)} \le U_{(2)} \cdots \le U_{(m)}$ are the order statistics of \iid uniform random variables $U_1, \ldots, U_m$ on $(0, 1)$.

Taking $k = 2$ in \eqref{eq:fdr_k_upperbound}, we proceed to obtain the bound on $\fdr_2$ in the following lemma. We note that because of some technical issues, the following analysis applies to the case of $k=2$. For general $k$, we use a different technique to bound $\fdr_k$. See Section~\ref{sec:fdr-contr-large}.
\begin{lemma}\label{lm:fdr_2_term}
There exists an absolute constant $C$ such that
\[
\E \left[ \max_{2 \le j \le m} \frac{qj}{m U_{(j)}} \right] \le Cq\,.
\]
\end{lemma}

To bound the above expectation, we use a well-known representation for the
uniform order statistics (e.g. see \cite{order}):
$(U_{(1)}, \ldots, U_{(m)}) \overset{d}{=} \left( T_1 /T_{m+1}, \ldots, T_m/ T_{m+1} \right)$,
where $T_j = \xi_1 + \cdots + \xi_j$, and $\xi_1, \ldots, \xi_{m+1}$ are \iid exponential random
variables.  Denote by $W_j = jT_{m+1}/T_j$. Then Lemma~\ref{lm:fdr_2_term} is equivalent to bounding
$\E \left( \max_{2 \le j \le m} W_j/m \right) \le C$.

Intuitively the maximum is more likely to be realized by smaller $j$'s
as increasing $j$ would increase the concentration of $T_j/j =
\sum_{i=1}^j \xi_i/j$.  Then above expectation can be bounded by
considering a few terms of $W_j$ for small $j$'s.  Indeed this
intuition can be made rigorous by observing that $W_1, \ldots,
W_{m+1}$ is a \emph{backward submartingale}, and hence we can use the
well known technique to bound $\E \max_{2\le j\le m} W_j$ by some
statistics of $W_2$, which we can then estimate. We present the
details as below.

\begin{lemma}\label{lm:submartingale}
With $T_j, W_j$ defined as above, we have that $W_1, \ldots, W_{m+1}$
is a backward submartingale with respect to the filtration (or
conditional on the ``history'') $\mathcal{F}_j=\sigma(T_j,
T_{j+1},\ldots,T_{m+1})$ for $j=1,\ldots,m+1$,
i.e. $\E(W_j|\mathcal{F}_{j+1}) \geq W_{j+1}$, for $j=1,\ldots,m$.
\end{lemma}

Now we apply Lemma~\ref{lm:submartingale} to prove Lemma~\ref{lm:fdr_2_term}.
\begin{proof}[Proof of Lemma \ref{lm:fdr_2_term}]
Using the exponential random variables representation, it suffices to prove the inequality in which $\E \left[ \max_{2 \le j \le m} \frac{qj}{m U_{(j)}} \right]$ is replaced by
\[
\frac{q}{m}\E \left[ \max_{2 \le j \le m} \frac{jT_{m+1}}{T_j}  \right] = q\E \left[ \max_{2 \le j \le m} \frac{W_j}{m} \right],
\]
where all the notations are followed from Lemma \ref{lm:submartingale}.
Given that $W_j/m$ is a backward submartingale by Lemma \ref{lm:submartingale}, we can apply Theorem 5.4.4 in
\cite{durrett}, which concludes
\begin{align*}
\E \Big( \max_{2 \le j \le m} W_j/m \Big) &\le (1 - \ep{-1})^{-1}\Big{[}1 + \E\Big{(}\frac{W_2}{m}\log \frac{W_2}{m}; \frac{W_2}{m} \geq 1\Big{)} \Big{]} \\
&= (1 - \ep{-1})^{-1}\Big{[}1 + \E\Big{(}\frac{2}{mU_{(2)}}\log \frac{2}{mU_{(2)}}; \frac{2}{mU_{(2)}} \geq 1\Big{)} \Big{]}\,.
\end{align*}
A bit of analysis reveals that the \rhs of the last display is bounded. Hence, to complete the proof it suffices to show that the \rhs of the above display is uniformly bounded for all $m$. The fact that $U_{(2)}$ is distributed as $\mathrm{Beta}(2, m-1)$ allows us to evaluate the expectation in the last display as
\begin{multline}\nonumber
\E\Big{(}\frac{2}{mU_{(2)}}\log \frac{2}{mU_{(2)}}; \frac{2}{mU_{(2)}} \geq 1\Big{)} = \int_0^{\frac2{m}}\frac{u(1-u)^{m-2}}{\mathrm{B}(2, m-1)}\frac2{mu}\log\frac2{mu}\d u \\\overset{v=mu}{=} \frac{m-1}{m}\int_0^22(1-v/m)^{m-2}\log\frac2{v} \d v \le \int_0^22\log\frac2{v} \d v\,,
\end{multline}
which is finite and independent of $m$.
\end{proof}

Provided the bound on $\fdr_2$, we can easily obtain the bound on
$\fdr$. Taking $k = 1$ in \eqref{eq:fdr_k_upperbound}, we get
\[
\mathrm{FDR} \le  \E \left[ \max_{2 \le j \le m}\frac{j}{\ceil{m U_{(j)}/q}} \right] + \E \left[ \min \left\{ \frac{1}{\ceil{m U_{(1)}/q}}, 1 \right\}  \right]\,,
\]
It is easy to show the second term is bounded by $q\log(1/q)$.
The lemma below controls the second term in the above display.
\begin{lemma}\label{lm:fdr_1_term}
\[
\E \left[ \min\Big{\{} \frac{1}{\ceil{m U_{(1)}/q}}, 1 \Big{\}} \right] \le q\log\frac1q + C_0 q\,,
\]
where $C_0 = 1+2/\sqrt{\ep{1}}\leq 2.3$. 
\end{lemma}

This proves the bound for FDR in Theorem~\ref{thm:robustfdr}. Without a lower bound on the number of discoveries, we suffer an extra additive term of $q\log(1/q)$.  Here comes an explanation how does this term emerge. Let $p^0_{\min}$ be the smallest $p$-value from all the $m_0$ true null hypotheses and, to be the most adversarial, all the false null $p$-values are set to zero. Then an adversary can reject $R - 1$ false null hypotheses, and then the true null hypothesis with $p$-value $p^0_{\min}$, where $R = \ceil{mp^0_{\min}/q}$. For large $m_0$, $E:=m_0p^0_{\min}$ is asymptotically distributed as exponential with mean 1. If $m_0 \approx m$, observe that
\[
\frac1{R} \approx \frac{q}{mp^0_{\min}} = \frac{qm_0}{mE} \approx \frac{q}{E}\,.
\]
The partial expectation of the last term $q/E$ on $[q, \infty)$ is approximately
\[
\int_{q}^{\infty}\frac{q}{u}\ep{-u} \d u = q\log\frac1{q} + O(q)\,.
\]
This justifies the logarithmic term $\log\frac1{q}$.

\subsection{Controlling $\fdr^k$} 
\label{sec:fdr-contr-large}
FDR control comes naturally for studies where a large number of rejections are expected to be made.  Provided
with this side information, we can improve the constant $C$ in
(\ref{eq:fdr2}) to $1$ asymptotically for $R \gg 1$. As mentioned
earlier, we will provide an upper bound on $\fdr^k$.

% % \begin{definition}\label{def:step_down_control}
% % Critical Values $(\alpha_1,
% % \ldots,\alpha_R)$ are said to
% % \emph{strongly dominate} $(p_{i_1}, p_{i_2}, \ldots, p_{i_R})$, if
% % $p_{i_{\pi(j)}} \le
% % \alpha_j, j = 1, \ldots, R$ for some permutation $\pi$. 

% % A multiple testing procedure with critical values $(\alpha_1, \ldots,
% % \alpha_m)$ is {\em strongly robust} for testing $H_1,
% % \ldots, H_m$, if the $p$-values of the rejected hypotheses 
% % $(p_{i_1}, p_{i_2}, \ldots, p_{i_R})$ are always strongly dominated by $(\alpha_1, \ldots, \alpha_R)$, where $R\in\{0, 1, \ldots, m\}$ is the number of rejections.
% % \end{definition}

\begin{theorem}\label{thm:stronglyrobustfdr}
Assume joint independence of the true null test statistics. Then any procedure adaptive to the \BHq critical values obeys
\begin{equation}\label{eq:fdrk}
\fdr^k \leq \left( 1+2/\sqrt{qk} \right) q \, .
\end{equation}
\end{theorem}

\begin{proof}[Proof of Theorem~\ref{thm:stronglyrobustfdr}]
Similar to the argument for bounding $\fdr$ and $\fdr_2$, it suffices to consider the \textit{most adversarial scenario}, where it always holds that
\[
V \le \#\Big{\{}i \in \mathscr{N}_0: p_i \le \frac{qR}{m}\Big{\}}\,,
\]
which gives an upper bound on FDP:
\begin{equation}\label{eq:high_power_basis}
\mathrm{FDP} \le \max_{R \le j \le m}\frac{\# \{ i \in \mathscr{N}_0: p_i \le qj/m \}} {j} \,.
\end{equation}
Similar to what has been argued previously, the above inequality still
holds if all the null $p$-values are $m$ \iid uniform on $(0, 1)$. This
observation leads to the proof of~(\ref{eq:fdrk}) by showing the
following inequality, again by applying tools from Martingale theory
\[
\E\left[ \max_{k \le j \le m}\frac{\#\big{\{}1 \le i \le m: U_i \le \frac{qj}{m}\big{\}}}{j} \right] \le \big(1 + 2/\sqrt{qk}\big)q\,.
\]

Denote by as usual $V_j = \#\{1 \le i \le m: U_i \le \frac{qj}{m}\}$ and $W_j = \frac{V_j}{j}$. Conditional on $W_{j+1}$, $V_{j+1}$ of $U_i$ are uniformly distributed on $[0, q(j+1)/m]$. Hence the conditional expectation of $V_j$ given by $W_{j+1}$ is 
\[
\E(V_j|W_{j+1}) = \frac{V_{j+1}\frac{qj}{m}}{\frac{q(j+1)}{m}} = \frac{jV_{j+1}}{j+1}\,,
\]
which is equivalent to
\[
\E(W_j|W_{j+1}) = W_{j+1}\,.
\]
Thus, $(W_j - q)_+$ is a backward submartingale, which allows us to use the martingale $\ell^2$ maximum inequality:
\[
\E \left[ \max_{k \le j \le m} (W_j-q)_+ \right]^2 \le \left( \frac{2}{2-1} \right)^2 \E (W_{k}-q)_+^2 \le 4\E(W_{k}-q)^2 = \frac{4q(1-qk/m)}{k} < \frac{4q}{k}\,.
\]
Last, Jensen's inequality gives
\[
\E \left[ \max_{k \le j \le m} W_j \right] \le q + \E \left[ \max_{k \le j \le m} (W_j - q)_+ \right] \le q + \sqrt{\E \left[ \max_{k \le j \le m} (W_j-q)_+ \right]^2} \le q + 2\sqrt{\frac{q}{k}}\,.
\]
\end{proof}

%%% Local Variables:
%%% mode: latex
%%% TeX-master: "paper"
%%% End:

\section{Private Adaptive Procedures}
\label{sec:privfdr}
One alternative interpretation of Theorem~\ref{thm:robustfdr} is that \BHsmall is robust with respect to small perturbation of $p$-values. That is if we add small enough noise to the $p$-values and then apply
\BHsmall procedure, the FDR can still be controlled within the bound given in the theorem. It turns out the additive sensitivity of $p$-values might be large, but the relative change is much smaller.

Now the important question is how to add noise.
In the literature, often additive noise is used to guarantee
privacy. But this would not work well for $p$-values as \BHsmall
considers the $p$-values in the range of $O(1/m)$ but a simple
analysis shows that the sensitivity of some standard $p$-values can be
as large as $1/\sqrt{n}$, which can be much larger than $1/m$.  It
turns out that while the additive sensitivity might be large, the
relative change is much smaller.
This motivates us to consider multiplicative sensitivity.  Hence we
will add multiplicative noise (or additive noise to the logarithm of
$p$-values).  Together with the private top-$k$ algorithm, this gives
us the private FDR algorithm.

For the ease of presentation, we assume that we are given an
upper bound $k$ of the number of rejections. The parameter $k$ should
be comparable to the number of true rejections and small compared to
$m$.  It is easy to adapt our algorithm to the case when $k$ is not
given, for example, by the standard doubling trick.  This only
incurs an extra logarithmic factor in our bound.

\subsection{Sensitivity of $p$-values}\label{sec:prelim:sensitivity}

Assume the input to our private FDR algorithm is an $m$-tuple of
$p$-values $p=(p_1,\dots,p_m)$ obtained by running $m$ statistical
tests on a dataset~$\dbax$.  Motivated by the normal approximation (or
related $\chi^2$ approximation) frequently used in computing $p$-values,
the change in a $p$-value caused by the addition or deletion of the data
of one individual is best measured multiplicatively; however, when the
$p$-value is very small the relative change can be very large. This
gives rise to the following definition of (truncated) multiplicative
neighborhood.
\begin{definition}[$(\eta,\nu)$-neighbors]
\label{def:multi_sensi}
Tuples $p=(p_1,\dots, p_m), p'=(p'_1,\dots,p'_m)$ are {\em $(\eta,\nu)$-neighbors} if, for all $1\leq i\leq m$,
either $p_i,p_i'<\nu$, or $\ep{-\eta} p_i \leq p'_i \leq \ep\eta p_i$. 
\end{definition}
%%\wjnote{Shall we call this $p$-value $(\eta, \nu)$-sensitive?}

%%%\cnote{%We must make clear that we will ensure that adjacent datasets lead to $(\eta,\nu)$-neighboring $m$-tuples.  
%We used to have some text about computing the $p$-values where we said that we bottom-coded them at $\nu$.  I don't see that right now.  
%We must also make clear what happens to very small $p$-values in our algorithm. This seems to be missing from this version of the paper?!}
%%%\linote{I have added some in the first paragraph. Is it OK}

The privacy of our FDR algorithm will be defined with respect to such
neighborhood. Next we will explain that some standard $p$-values
computed on neighboring databases are indeed $(\eta,\nu)$-neighbors
for small $\eta$ and $\nu$.  We will give an intuitive explanation
using Gaussian approximation but omit the proof details. Consider a
database consisting of the records of $n$ people and a hypothesis $H$
to be tested, where each person contributes a statistic $t_i, i = 1,
\ldots, n$ with $|t_i|\leq B$ (for example, $t_i$ is the number of
minor alleles for a given SNP.) In many interesting cases, the
sufficient statistic for testing the hypothesis is $T = t_1 +
\cdots + t_n$, and under the null hypothesis $H$, each $t_i$ has mean
$\mu$ and variance $\sigma^2$.  Then $(T-n\mu)/(\sqrt{n}\sigma)$ is
asymptotically distributed as standard normal variable.  Assuming $T$
tends to be larger under the alternative hypothesis, we can
approximately compute $p$-value $p(T)=\Phi(-(T-n\mu)/(\sqrt{n}\sigma))$,
where $\Phi$ is the cumulative distribution function of standard Gaussian.  Consider a neighboring
database where person $i$ is replaced, so $T'-T=t'_i-t_i$. Writing
$p'=p(T')$ and invoking that $\Phi(-x) \approx \frac1{x}\phi(x) =
\frac1{x\sqrt{2\pi}}\ep{-\frac{x^2}{2}}$ for large $x > 0$, 
we can show that $p(T)$ is $(\eta,\nu)$ multiplicative sensitive for
$\eta \approx B\sqrt{2\log(1/\nu)/n}/\sigma = O(\sqrt{\log(1/\nu)/n}/\sigma)$. 

% More generally, we have
% \red{For independent bounded statistics of $n$ people, the $p$-value is
% $(\eta,\nu)$ multiplicatively sensitive with respect to any individual
% change for $\eta = O(\sqrt{\log(1/\nu)/n})$.}

\subsection{A private FDR algorithm}
To achieve privacy with respect to $(\eta,\mu)$-neighborhood, we apply
\peel with properly scaled noise to the logarithm of the input
$p$-values (see Step~1 of Algorithm~\ref{algo:pbh2}). Denote by $\Delta_k = (1+o(1))\sqrt{k\log(1/\delta)}\log m/\e$ the accuracy bound provided by \peel.  Then we run \BHsmall with cutoff values set as $\alpha_j'=\log (\alpha_j+\nu) + \eta \Delta_k$. The details are
described in Algorithm~\ref{algo:pbh2}.
\begin{algorithm}[ht]
\caption{Differentially private FDR-controlling algorithm}\label{algo:pbh2}
\begin{algorithmic}[1]
\REQUIRE $(\eta, \nu)$-sensitive $p$-values $p_1, \cdots, p_m$ and $k\geq 1$ and $\eps,\delta$
\ENSURE a set of up to $k$ rejected hypotheses
\STATE for each $1\leq i\leq m$, set $x_i = \log (\max\{ p_i,\nu \})$
\STATE apply \peel to $x_1,\ldots, x_m$ with noise $\lap(\eta \sqrt{k\log(1/\delta)}/\e)$ to obtain $(i_1,y_1), \ldots, (i_k,y_k)$
\STATE apply \BHsmall{} to $y_1,\ldots, y_k$ with cutoffs $\alpha'_j$ for $1\leq j\leq k$ and reject the corresponding hypotheses
\end{algorithmic}
\end{algorithm}

Our main result is that this algorithm is differentially private and controls FDR, while maintaining a descent power. In particular, the second and third conclusions serve as, respectively, formal versions of Theorems~\ref{thm:PrivateFDRInformal} and \ref{thm:PrivatePowerInformal}.
\begin{theorem}\label{thm:main}
Algorithm~\ref{algo:pbh2} obeys:
\begin{enumerate}
\item[(a)] It is $(\e,\delta)$-differentially private;
\item[(b)] If \BHsmall rejects $k'\leq k$ hypotheses, Algorithm~\ref{algo:pbh2} rejects at least $k'$ hypotheses with probability $1-o(1)$;
\item[(c)] Suppose $\nu=o(1/m)$, then the FDR of Algorithm~\ref{algo:pbh2} satisfies the bounds in Theorem~\ref{thm:robustfdr} with $q$ replaced by  $\ep{\eta \Delta_k}(1+o(1)) q$.
\end{enumerate}
%\wjnote{the monotonic alternative condition hasn't been introduced yet. What is $\Delta$? Shall we adjust the cutoffs to let $q$ remain fixed?}
\end{theorem}

\begin{proof}
\begin{enumerate}
\item[(a)] For two $(\eta,\nu)$ neighbor $p,p'$, by definition for each $i$,
either $p_i,p_i'\leq \nu$ or $e^{-\eta} p_i \leq p_i' \leq e^{\eta}
p_i$.  In either case, we have that $|x_i' - x_i|\leq \eta$. Hence the
privacy of Algorithm~\ref{algo:pbh2} follows from the privacy of
\peel.

\item[(b)] By Theorem~\ref{thm:peeling}, we have that with probability $1-o(1)$,
$y_j\leq x_{(j)}+\eta \Delta_k$. Hence if for $1\leq j\leq k'$,
$p_{(j)}\leq \alpha_j$, then we have $x_{(j)}\leq \log
(\alpha_j+\nu)$. Thus $y_j\leq \log (\alpha_j+\nu) + \eta\Delta_k =
\alpha_j'$. Hence Algorithm~\ref{algo:pbh2} rejects at least $k'$ hypotheses as well. 

\item[(c)] Suppose that the algorithm rejects the set of hypotheses $t_1,
\ldots, t_{k'}$, then we have that $y_{t_j} \leq \alpha_j'$. By Theorem~\ref{thm:peeling} with high probability the corresponding $p$-value is bounded by $\ep{O(\eta \Delta_k)} (\alpha_j+\nu)$.  By that $\alpha_j = jq/m$, Algorithm~\ref{algo:pbh2} is an adaptive procedure with respect to the cutoff
\[
(q'/m,2q'/m,\ldots,kq'/m),
\]
where $q' = \ep{O(\eta\Delta_k)}(1+o(1))q$. Then the bounds in Theorem~\ref{thm:robustfdr} apply by plugging in $\ep{O(\eta\Delta_k)}(1+o(1))q$ in place of $q$.
\end{enumerate}
\end{proof}

For the binary database with independent statistics, take $\eta=O(\sqrt{\log m/n})$ and $\nu=1/m^2$. Then, from the discussion in Section~\ref{sec:prelim:sensitivity} we get
\begin{corollary}\label{cor:main}
Algorithm~\ref{algo:pbh2} is $(\e,\delta)$-differentially private. In addition, if $k\ll n/\log^{O(1)}(m)$, with probability $1-o(1)$, it rejects at least the same amount of hypotheses as in Algorithm~\ref{algo:smallbh} and controls FDR at $q\log(1/q) + C(1+o(1))q$, $\fdr_2$ at $C(1+o(1))q$, and $\fdr_k$ at $(1+o(1))q$.
\end{corollary}

One drawback of the above algorithm is that it needs to run the peeling algorithm which takes $\wO(km)$ time. This can be expensive
for  large $k$.  In the next section, we show the privacy of
one-shot algorithm $\oneshot$. By replacing
\peel with $\oneshot$ in Algorithm~\ref{algo:pbh2}, we can obtain the
\emph{Fast Private FDR algorithm} which has essentially the same quality bound but
with running time $\wO(k+m)$.

%%% Local Variables:
%%% mode: latex
%%% TeX-master: "paper"
%%% End:

\section{One-shot mechanism for reporting top-$k$}
\label{sec:mink}

In the one-shot mechanism $\oneshot$, we add noise $\lap(\lambda)$
once to each value and report the indices of the minimum-$k$ noisy
values (Algorithm~\ref{algo:oneshot}). Here the input to our algorithms is $m$ counts $x = (x_1,\ldots,x_m)$ which are the marginals of database $D$. So two set of counts $x$ and $x'$, representing respectively the marginals of two neighboring databases $D$ and $D'$, satisfy $\|x-x'\|_\infty \leq 1$.

Using a coupling argument, it is easy to show by setting $\lambda =
O(k/\e)$, the one-shot mechanism is $(\e,0)$-differentially private.
\begin{lemma}\label{lem:oneshotk}
Algorithm~\ref{algo:oneshot} is $(\e,0)$-differentially private if we set $\lambda = 2k/\e$. 
\end{lemma}
\begin{proof}[Proof of Lemma \ref{lem:oneshotk}]
Let $x = (x_1, \ldots, x_m)$ and $x' = (x'_1, \ldots, x'_m)$ be adjacent, i.e., $\|x - x' \|_{\infty} \le 1$. Let $S$ be an arbitrary $k$-subset of $\{1, \ldots, m\}$. Let $\mathcal{G}$ consist of all $(g_1, \ldots, g_m)$ such that $\oneshot$ on $x$ reports $S$. Similarly we have $\mathcal{G}'$ for $x'$. It is clear to see that $\mathcal{G} - 2\cdot \vec{1}_S \subset \mathcal{G}'$, where $\vec{1}_S \in \{0, 1\}^m$ satisfies $\vec{1}_S(i) = 1$ if and only if $i \in S$. Hence we get
\begin{multline}\nonumber
\prob(\oneshot(x) = S) = \int_{\mathcal{G}}\frac1{2^m\lambda^m}\ep{-\frac{\|g\|_1}{\lambda}}\d g \ge \int_{\mathcal{G}' + 2\cdot\vec{1}_S}\frac1{2^m\lambda^m}\ep{-\frac{\|g\|_1}{\lambda}}\d g\\
= \int_{\mathcal{G}'}\frac1{2^m\lambda^m}\ep{-\frac{\|g + 2\cdot\vec{1}_S\|_1}{\lambda}}\d g \ge \int_{\mathcal{G}'}\frac{\ep{-2k/\lambda}}{2^m\lambda^m}\ep{-\frac{\|g\|_1}{\lambda}}\d g  = \ep{-\e}\prob(\oneshot(x') = S)\,.
\end{multline}
On the opposite side, we have $\prob(\oneshot(x) = S) \le \ep{\e}\prob(\oneshot(x') = S)$, which completes the proof since $S$ is arbitrary.
\end{proof}
Our goal is to reduce the dependence on~$k$ to~$\sqrt{k}$ for approximate differential privacy.

% \remove{
% Compared to the peeling procedure, the above bound is sub-optimal when
% $k$ is large. As a main contribution of the paper, we can show that
% the mechanism is still private even if we only add $\lambda =
% \widetilde{O}(\sqrt{k})$ noise.
% }
\begin{theorem}\label{thm:oneshot}
Take $\e \leq \log(m/\delta)$. There exists universal constant $C>0$
such that if we set $\lambda=C\sqrt{k\log(m/\delta)}/\e$, then
$\oneshot$ is $(\e,\delta)$-differentially private. In addition, with probability
$1-o(1)$, for every $1\leq j\leq k$, $x_{i_j}-x_{(j)} = O(\sqrt{k\log(m/\delta)}\log m/\e)$. 
\end{theorem}

Unlike in the peeling approach, in which an ordered set is returned,
in $\oneshot$, only a subset of $k$ elements, but not their ordering,
is returned.  The privacy proof crucially depends on this. If we would
also like to report the values, we can report the noisy values by adding random noise \emph{freshly} sampled from $\lap(O(\sqrt{k\log(1/\delta)}/\e))$
to each value of those $k$ elements.
\begin{algorithm}[ht]
\caption{The one-shot procedure $\oneshot$ for privately reporting minimum-$k$ elements}\label{algo:oneshot}
\begin{algorithmic}[1]
\REQUIRE $m$ values $x_1, \cdots, x_m$, $k\geq 1$, $\lambda$
\ENSURE $i_1, \ldots, i_k$
\FOR {$i=1$ to $m$}
  \STATE set $y_i = x_i+g_i$ where $g_i$ is sampled
  i.i.d. from $\lap(\lambda)$
\ENDFOR
\STATE sort $y_1,\ldots, y_m$ from low to high, $y_{(1)} \le y_{(2)} \le \cdots \le y_{(m)}$
\STATE return the set $\{(1), \ldots, (k)\}$
\end{algorithmic}
\end{algorithm}

In Theorem~\ref{thm:oneshot}, the quality bound on $x_{i_j} - x_{(j)}$ is an immediate consequence from the properties of exponential random variables. The claim on the privacy part is equivalent to the lemma below.
\begin{lemma}\label{lm:oneshot_private}
For any $k$-subset $S$ of $\{0, 1\}^m$ and any adjacent $x, x'$, we have 
\[
\P(\oneshot(x) \in S) \le \ep{\e}\P(\oneshot(x')\in S) + \delta\,.
\]
\end{lemma}

We divide the event space by fixing the $k$th smallest
noisy element, say $j$, together with the noise value, say $g_j$. For
each partition, whether an element $i\neq j$ is selected by $\oneshot$
only depends on whether $x_i+g_i \leq x_j+g_j$, which happens with
probability $q_i = G((x_j + g_j - x_i)/\lambda)$. Here $G$ denotes
the cumulative distribution function of the standard Laplace distribution.  Hence, we consider the
following mechanism $\M$ instead: given $(q_1, \ldots, q_m)$ where
$0<q_i<1$, output a subset of indices where each $i$ is included with
probability $q_i$. In the following, we will first understand the
sensitivity of $q_i$ dependent on the change of $x_i$ and then show
that $\M$ is ``private'' with respect to the corresponding sensitivity
on $q$.
\begin{lemma}\label{lem:close}
For any $z,z'$, where $|z'-z|\leq 1$, $|G(z')-G(z)|\leq 2
|z'-z| G(z)(1-G(z))$.
\end{lemma}

\begin{definition}[$c$-closeness for vectors]
For two vectors $q=(q_1,\ldots, q_m)$ and $q'=(q'_1,\ldots,q'_m)$, we
say $q,q'$ are {\em $c$-close} if for each $1\leq i\leq m$, $|q_i-q'_i|\leq
c q_i(1-q_i)$. 
\end{definition}
The following is the crucial lemma whose proof requires much technical work.
\begin{lemma}\label{lm:dp_q_new}
Assume $\e\leq\log(1/\delta)$ and $k\geq\log(1/\delta)$. There exists
constant $C_1>0$ such that if $q$ and $q'$ are $c$-close with $c \leq
\frac{\e}{C_1\sqrt{k\log(1/\delta)}}$, then for any set $S$ of
$k$-subsets of $\{0,1\}^m$,
$\P(\M(q)\in S) \le \ep{\e}\P(\M(q')\in S) + \delta$.
\end{lemma}

% Lemma~\ref{lem:close} and~\ref{lm:dp_q_new}
We first show how these two lemmata imply Lemma~\ref{lm:oneshot_private} and then prove Lemma~\ref{lm:dp_q_new}.
\begin{proof}[Proof of Lemma~\ref{lm:oneshot_private}]
If $k\leq \log(m/\delta)$, then by Lemma~\ref{lem:oneshotk}, the
mechanism is $(\e,0)$ private. Assume $k\geq \log(m/\delta)$.  Denote
by $I_k$ the random variable of the $k$th smallest element in terms
of the noisy count. For any given $I_k=j$ and the noise $g_j = g$, we have
that $\P(i\in\oneshot(x)) = G((x_j + g - x_i)/\lambda):=q_i$. Set
$q = q(g) = (q_i)$ for $1\leq i\leq m$ and $i\neq j$. Write $S_j=\{s/\{j\}\;:\;
s\in S\mbox{ and }j\in s\}$. Then
\[
\P(\oneshot(x)\in S, I_k = j | g_j = g) = \P(M(q)\in S_j).
\]
Since $\|x-x'\|_{\infty}\leq 1$, we have that for
any $i$, $|(x_j+ g -x_i)/\lambda-(x'_j+ g -x'_i)/\lambda|\leq
2/\lambda$. By Lemma~\ref{lem:close}, $q$ and $q'$ are $4/\lambda$
close. Set $\lambda \geq 4C_1\sqrt{k\log(m/\delta)}/\e$, then $q$ and
$q'$ are $\frac{\e}{C_1\sqrt{k\log(m/\delta)}}$-close. Applying
Lemma~\ref{lm:dp_q_new} to $q$ and $q'$ with parameters $\e$ and
$\delta/m$, we have
\[\P(\oneshot(x)\in S, I_k = j | g_j = g)\leq \ep{\e} \P(\oneshot(x')\in S, I_k = j | g_j = g) + \delta/m\,.\]
\end{proof}

Now we prove Lemma~\ref{lm:dp_q_new}. Since $S$ consists of $k$-sets,
we first show that if $\sum_i q_j \gg k$, then $\P(\M(q)\in S)$ is
small. This can be done by applying the standard concentration bound. Write $K = (1+2\sqrt{\log(m/\delta)/k})k \leq 3k$ (recall we assume that $k\geq \log(m/\delta)$).

\begin{lemma}\label{lm:poisson_binomial}
Let $Z_1,\ldots, Z_m$ be $m$ independent Bernoulli random variables
with $\P(Z_i = 1) = q_i$. Suppose
\[\sum_{i=1}^m q_i \ge (1 + t)k\,,\]
for any $t > 0$. Then
\[\P(\sum_i Z_i \leq k) \le \exp\left( -(1+t)h(t/(1+t))k\right)\,,\]
where $h(u) = (1+u)\log(1+u) - u$. Consequently, by setting
$t=2\sqrt{\log(m/\delta)/k}$, we have that if $\sum_{i=1}^m q_i \ge
K$, then $\P(\sum_i Z_i \leq k)\le \delta$.
\end{lemma}

The above lemma follows immediately from the classical Bennett's inequality stated as follows.
\begin{lemma}[Bennett's inequality, \cite{bennett}]
\label{lm:bennett_ineq}
Let $Z_1 \ldots, Z_n$ be independent random variables with all means being zero. In addition, assume $|Z_i| \le a$ almost surely for all $i$. Denoting by
\[
\sigma^2 = \sum_{i=1}^n\mathrm{Var}(Z_i),
\]
we have
\[
\P\left(\sum_{i=1}^n Z_i > t \right) \le \exp\left(-\frac{\sigma^2h(at/\sigma^2)}{a^2}\right)
\]
for any $t \ge 0$, where $h(u) = (1+u)\log(1+u) - u$.
\end{lemma}

With Lemma~\ref{lm:poisson_binomial}, we only need to consider the
case of $\sum_i q_i \leq K$.
But this is a more difficult case.
We first represent a set $s\subseteq\{1,\ldots,m\}$ by a binary vector
$z\in \{0, 1\}^m$. Write $\P_{q}(z)$ as
\begin{equation}\label{eq:p_expression}
\P_q(z) = \prod_{i:z_i = 1}q_i\prod_{i:z_i=0}(1-q_i)\,.
\end{equation}
Now our goal is to show that for any $S$ consisting of weight $k$
vectors in $\{0,1\}^m$,
\[\sum_{z\in S} \P_q(z) \leq \ep{\e}\sum_{z\in S}\P_{q'}(z) + \delta\,.\]
Denote by 
\[S^\ast = \{z\;:\; \P_q(z) \geq \ep{\e}\P_{q'}(z)\}\,.\]
The proof of Lemma~\ref{lm:dp_q_new} would be completed by showing
\begin{equation}\nonumber%%%\label{eq:main}
\sum_{z\in S^\ast} \P_q(z) \leq \delta\,.
\end{equation}
By \eqref{eq:p_expression}, $x\in S^\ast$ if and only if that
\begin{equation}\label{eq:diff}
\prod_{i:z_i = 1}q_i\prod_{i:z_i=0}(1-q_i) \geq \ep{\e}\prod_{i:z_i = 1}q'_i\prod_{i:z_i=0}(1-q'_i)\,.
\end{equation}
Write $\Delta_i = q'_i - q_i$.  By the $c$-closeness of $q$ and $q'$,
we have that $|\Delta_i|\leq cq_i(1-q_i)$.  Taking logarithm of the
two sides and rearranging, (\ref{eq:diff}) is equivalent to
\[\sum_{i:z_i=1} \log(1+\Delta_i/q_i) + \sum_{i:z_i=0}\log(1 - \Delta_i/(1-q_i))\leq -\e\,.\]
To bound $\sum_{z\in S^\ast} \P_q(z)$, consider the independent $0,1$
random variables $Z_1, \ldots, Z_n$, where for each $i$, $Z_i = 1$
with probability $q_i$. Define
\[\zeta_i = Z_i\log(1+\Delta_i/q_i) + (1 - Z_i)\log(1 - \Delta_i/(1-q_i))\,.\]
Let $I$ be the indicator function. Then we have
\[\sum_{z\in S^\ast} \P_q(z)=\sum_z I\left(\sum_i \zeta_i \leq -\e\right) \P_q(z) = \P\left(\sum_i \zeta_i \leq -\e\right)\,,\]
where the last probability is over the distribution of
$Z_1,\ldots,Z_m$.

The rest of the proof is to prove that $\P(\sum_i \zeta_i \leq
-\e)\leq \delta$.  It is easy to check that $\zeta_1 + \ldots +
\zeta_m$ has mean
\[\sum_{i=1}^m \left( q_i\log(1+\Delta_i/q_i) + (1-q_i)\log(1-\Delta_i/(1-q_i))\right)\]
and variance
\[\sigma^2 = \sum_{i=1}^m q_i(1-q_i)\log^2\frac{1+\Delta_i/q_i}{1-\Delta_i/(1-q_i)}\,.\]
Before applying Bennett's inequality, we check the ranges of the centered random variables $\zeta_i - q_i\log(1+\Delta_i/q_i) -(1-q_i)\log(1-\Delta_i/(1-q_i))$, which are bounded in absolute value by
\[\max_{1 \le i \le m}\left|\log\frac{1 + \Delta_i/q_i}{1-\Delta_i/(1-q_i)}\right|\,.\]
Then Bennett's inequality asserts that for any $t \ge 0$,
\[\sum_{i=1}^m \zeta_i \ge \sum_{i=1}^m \left( q_i\log(1+\Delta_i/q_i) + (1-q_i)\log(1-\Delta_i/(1-q_i))\right) - t\]
with probability at least
\[1 - \exp\left(-\sigma^2h(At/\sigma^2)/A^2 \right)\,,\]
where $A=\max_{1 \le i \le m}|\log\frac{1 +
\Delta_i/q_i}{1-\Delta_i/(1-q_i)}|$. Then, taking
\[t = \e + \sum_{i=1}^m \left( q_i\log(1+\Delta_i/q_i) + (1-q_i)\log(1-\Delta_i/(1-q_i))\right)\,,\]
Bennett's inequality concludes that
\[\P(\sum_i \zeta_i\leq -\e) \le \exp\left(-\sigma^2h(At/\sigma^2)/A^2 \right)\,.\]
Hence, the case of $\sum_i q_i \leq K$ can be established by proving
\begin{equation}\label{eq:exp_bound_delta}
\frac{\sigma^2h(At/\sigma^2)}{A^2}  \ge \log\frac1{\delta}\,.
\end{equation}
The proof of \eqref{eq:exp_bound_delta} is deferred to the appendix. This proves Theorem~\ref{thm:oneshot} via first completing the proof of Lemma~\ref{lm:dp_q_new}.

%%% Local Variables:
%%% mode: latex
%%% TeX-master: "paper"
%%% End:

\section{Discussion}
\label{sec:discussion}

%\cnote{the remark below will be moved to a discussion section, probably at the end of the paper}

In addition to all the usual scientific reasons for investigating FDR control, our interest in the problem was piqued by the following observation.  As suggested by the Fundamental Law of Information Recovery, exposure of data to a query causes some amount of privacy loss that eventually accumulates until privacy is lost.  This leads to the notion of a privacy budget, in which a limit on what is deemed to be a reasonable amount of privacy loss is established, and access to the data is terminated once the cumulative loss reaches this amount. This denial of access can be difficult to accept.  

In a sense the multiple comparisons problem describes a different way in which data can be ``used up,'' to wit, in accuracy. This connection between multiple hypothesis testing and differential privacy provided the seeds for an exciting recent result showing that differential privacy
protects against false discoveries due to {\em adaptive} data analysis~\cite{DworkFHPRR14}. Although often confounded, adaptivity is a different source of false discovery than multiple hypothesis testing; resilience to adaptivity does not address the problem studied in the current paper.

%%% Local Variables:
%%% mode: latex
%%% TeX-master: "paper"
%%% End:

\section*{Acknowledgements}
Parts of the research were performed when W.~S.~was a research intern at Microsoft Research Silicon Valley. We would like to thank David Siegmund, Emmanuel Cand\`es, and Yoav Benjamini for helpful discussions.

\bibliographystyle{abbrv}
\bibliography{ref}

\appendix
\section{Technical Proofs}
\label{sec:technical-proofs}

\begin{proof}[Proof of Lemma~\ref{lm:submartingale}]
The proof is similar to Example 5.6.1 in \cite{durrett}.  By scaling,
assume that $\xi_i$ are exponential random variables with parameter 1,
i.e, $\E \xi_i = 1$. Note that $W_j$ is measurable with respect to
$\mathcal{F}_j$. In the proof, we first consider the conditional
expectation $\E( W_j^{-1}|\mathcal{F}_{j+1})$, then return to $\E(
W_j|\mathcal{F}_{j+1})$ by applying Jensen's inequality. Specifically, note that
\[
\E \left[ \frac{\xi_{l}}{j T_{m+1}}\big{|}\mathcal{F}_{j+1} \right] = \frac{1}{j T_{m+1}} \E(\xi_l |\mathcal{F}_{j+1}),
\]
since $T_{m+1}$ is measurable in $\mathcal{F}_{j+1}$. Next, observe that by symmetry we get
\[
\E ( \xi_{l}|\mathcal{F}_{j+1}) = \E ( \xi_{k}|\mathcal{F}_{j+1})
\]
for any $l, k \le j + 1$. Combining the last two displays gives
\begin{align*}
\E\left[ \frac{T_j}{j T_{m+1}}\big{|}\mathcal{F}_{j+1} \right] &= \E \left[ \frac{T_j}{(j+1) T_{m+1}} \big{|}\mathcal{F}_{j+1} \right] + \sum_{l=1}^j \E \left[ \frac{\xi_l}{j(j+1) T_{m+1}}\big{|}\mathcal{F}_{j+1} \right] \\
&= \E \left[ \frac{T_j}{(j+1) T_{m+1}}\big{|}\mathcal{F}_{j+1} \right] + \sum_{l=1}^j \E \left[ \frac{\xi_{j+1}}{j(j+1) T_{m+1}}\big{|}\mathcal{F}_{j+1} \right]\\
&= \E \left[ \frac{T_{j+1}}{(j+1) T_{m+1}}\big{|}\mathcal{F}_{j+1} \right] = \frac{T_{j+1}}{(j+1) T_{m+1}}\,.
\end{align*}
To complete, note that Jensen's inequality asserts that
\[
\E( W_j|\mathcal{F}_{j+1}) \geq \frac{1}{\E( W_j^{-1}|\mathcal{F}_{j+1})} = \frac{(j+1) T_{m+1}}{T_{j+1}} = W_{j+1}\,,
\]
as desired.
\end{proof}

\begin{proof}[Proof of Lemma~\ref{lm:fdr_1_term}]
%%%\cnote{I can't understand this sentence.  What is $u$ in the definition of the density of $\mathrm{Beta}(1,m)$?} \wjnote{I've changed it to the more standard $x$.}
The order statistic $U_{(1)}$ follows from $\mathrm{Beta}(1, m)$, whose density is $m(1-x)^{m-1}$ at the value $x$. Thus,
\begin{align*}
\E \left[ \min \left\{ \frac{1}{\ceil{m U_{(1)}/q}}, 1 \right\} \right] &\le \E \left[ \min \left\{ \frac{q}{m U_{(1)}}, 1 \right\} \right]\\
& = \P\left( U_{(1)} < q/m \right) + \int_{\frac{q}{m}}^1 \frac{q(1-x)^{m-1}}{x} \d x\,.
\end{align*}
In this display, the term $\P(U_{(1)} < q/m)$ is simply equal to $1 - (1-q/m)^m \le q$, and the integral satisfies
\begin{align*}
\int_{\frac{q}{m}}^1 \frac{q(1-x)^{m-1}}{x} \d x &\le q\int^m_q \frac{(1-y/m)^{m-1}}{y} \d y \\
&\le q\int^1_q \frac1{y} \d y + q\int^m_1 \ep{-y/2} \d y \\
&\le q\log\frac1{q} + \frac{2q}{\sqrt{\ep{1}}}\,.
\end{align*}
\end{proof}

\begin{proof}[Proof of Lemma \ref{lm:poisson_binomial}]
Bennett's inequality asserts
\[
\P(\sum_i Z_i \leq k) = \P\left(\sum_i Z_i - \sum_i q_i \leq -tk\right) \le \ep{-\sigma^2 h(tk/\sigma^2)}, 
\]
where $\sigma^2 = \sum_{i=1}^m q_i(1-q_i) \le (1+t)k$. Invoking the monotonically decreasing property of $\sigma^2 h(tk/\sigma^2)$ as a function of $\sigma^2$, we get
\[\P(\sum_i Z_i \leq k) \le \ep{-\sigma^2 h(tk/\sigma^2)} \le \exp\left( -(1+t)h(t/(1+t))k\right)\,.\]
\end{proof}

%%%%%%%%%%%%%%%%%%%%%%%%%
We now aim to bound $t, A$, and $\sigma$ using the fact $|\Delta_i| \le c
q_i(1-q_i)$. First, using $u-u^2 \leq \log(1+u)\leq u$ for $|u|\leq
1/2$, we have that
\begin{align*}
&\left|\sum_{i=1}^m \left( q_i\log(1+\Delta_i/q_i) + (1-q_i)\log(1-\Delta_i/(1-q_i))\right)\right|\\
\leq{}&\sum_i \left|q_i(\Delta_i/q_i + (\Delta_i/q_i)^2)+(1-q_i)(-\Delta_i/(1-q_i) + (\Delta_i/(1-q_i))^2)\right|\\
={}&\sum_i (q_i (\Delta_i/q_i)^2 + (1-q_i)(\Delta_i/(1-q_i))^2)\\
={}&\sum_i (q_i (\Delta_i/q_i)^2 + (1-q_i)(\Delta_i/(1-q_i))^2)\\
\leq{}&\left(c^2\sum_i(q_i(1-q_i)^2+q_i(1-q_i)^2\right)\quad\mbox{by $|\Delta_i|\leq c q_i(1-q_i)$}\\
\leq{}&c^2 \sum_i q_i \leq O(c^2 k)\quad\mbox{by $\sum_i q_i\leq K\leq 3k$}\,.
\end{align*}
Hence, we get
\begin{equation}\label{eq:t}
|t-\e|=O(c^2k)
\end{equation}
and
\[
\left|\log\frac{1+\Delta_i/q_i}{1-\Delta_i/(1-q_i)}\right|\leq 
\left|\log\frac{1+c(1-q_i)}{1-cq_i}\right| \leq \frac{c}{1-cq_i}=O(c).
\]
Combining these results gives
\begin{equation}\label{eq:A}
A = O(c)\,.
\end{equation}
In a similar way, we have
\begin{align}
\sigma^2 &= \sum_{i=1}^m q_i(1-q_i)\log^2\frac{1+\Delta_i/q_i}{1-\Delta_i/(1-q_i)}\nonumber\\
&\leq \sum q_i(1-q_i)O(c^2)\nonumber\\
&=O(c^2 k)\,.\label{eq:sigma}
\end{align}
Since $uh(a/u)$ is a decreasing function in $u$, from \eqref{eq:sigma} it follows that for any sufficiently large $C_2$
\[
\sigma^2h(At/\sigma^2)\geq C_2c^2k h(At/(C_2c^2k))\,,
\]
which, together with \eqref{eq:t} and \eqref{eq:A}, gives
\[At/(C_2c^2k) = O(c(\e+c^2k)/(C_2c^2 k)) = \frac1{C_2}O(\e/(ck)+c))\,.\]
For any sufficiently large $C_1$, set $c=\e/(C_1\sqrt{k\log(1/\delta)})$ and choose a sufficiently large constant $C_2$. Recognizing $k\geq \log(1/\delta)$, we have
\[At/(C_2c^2K) \leq 0.5\,.\]
Making use of the fact that $h(u)=\Omega(u^2)$ for $u\leq 0.5$, we have that
\[\frac{\sigma^2h(At/\sigma^2)}{A^2}\geq C_2 c^2k h(At/(C_2c^2k))/A^2 =\Omega(t^2/(c^2 k))\,.\]
Owing to $\e\leq\log(1/\delta)$ and $c=\e/(C_1\sqrt{k\log(1/\delta)})$, we set $C_1$ large enough such that $t=\Omega(\e)$. Consequently, we get
\[\frac{\sigma^2h(At/\sigma^2)}{A^2} =\Omega(\e^2/(c^2 k)) = \Omega(\log(1/\delta))\,.\]
This proves (\ref{eq:exp_bound_delta}) and completes the proof of
Lemma~\ref{lm:dp_q_new}. Thus, Theorem~\ref{thm:oneshot} is proved.

%%% Local Variables:
%%% mode: latex
%%% TeX-master: "paper"
%%% End:

\end{document}